\documentclass{article}
\usepackage[utf8]{inputenc}
\usepackage[left=1in,top=1in,right=1in,bottom=1in]{geometry}
\usepackage[linktocpage=true]{hyperref}
\usepackage{setspace}
\usepackage{amssymb, amsmath, amsthm, graphicx,mathrsfs}
\usepackage{caption,cite}
\usepackage{mathtools,color}

\usepackage{cleveref}
\usepackage{tikz,verbatim}
\usetikzlibrary{shapes,snakes}
\usetikzlibrary{arrows}

\newtheorem{thm}{Theorem}[section]

\newtheorem{defn}{Definition}

\newtheorem{conj}{Conjecture}
\newtheorem{prop}[thm]{Proposition}

\newtheorem{ques}[conj]{Question}

\newcommand{\Z}{\mathbb{Z}}

\renewcommand{\l}{\left}
\renewcommand{\r}{\right}

\title{On tight tree-complete hypergraph Ramsey numbers}
\author{Jiaxi Nie\footnote{School of Mathematics, Georgia Institute of Technology,
Atlanta, GA 30332 USA. {\tt jnie47@gatech.edu}.}}
\date{\today}

\begin{document}

\maketitle

\begin{abstract}
Chvátal showed that for any tree $T$ with $k$ edges the Ramsey number $R(T,n)=k(n-1)+1$ (``Tree‐complete graph Ramsey numbers.'' Journal of Graph Theory 1.1 (1977): 93-93). For $r=3$ or $4$, we show that, if $T$ is an $r$-uniform non-trivial tight tree, then the hypergraph Ramsey number $R(T,n)=\Theta(n^{r-1})$. The 3-uniform result comes from observing a construction of Cooper and Mubayi. The main contribution of this paper is the 4-uniform construction, which is inspired by the Cooper-Mubayi 3-uniform construction.
\end{abstract}

\section{Introduction}

Let $F$ be an $r$ uniform hypergraph ($r$-graph). An independent set of $F$ is a set of vertices containing no edge of $F$. The {\em independence number} of $F$, denoted $\alpha(F)$, is the maximum size of an independent set of $F$. The {\em Ramsey number} $R(F,n)$ is the minimum integer $N$ such that every $F$-free $N$-vertex $r$-graph has an independent set of size $n$. Determining or estimating $R(F,n)$ is a central topic in extremal combinatorics. 

Chavátal~\cite{chvatal1977tree} showed that for any tree $T$ with $k$ edges the Ramsey number $R(T,n)=k(n-1)+1$. In this short paper, we extend this result to hypergraph \emph{tight tree}.

\begin{defn}
An $r$-graph $T$ is a {tight $r$-tree} if its edges can be ordered as $e_1,\dots,~e_t$ so that 
\begin{equation}\label{eq:tight tree}
\forall i\ge 2~\exists v\in e_i~and~1\le s\le i-1~such~that~v\not\in\cup_{j=1}^{i-1}e_j~and~e_i-v\subset e_s.  
\end{equation}
\end{defn}

A hypergraph $H$ is \emph{non-trivial} if no vertex is contained in all of its edges, i.e. $\cap_{e\in H}e=\emptyset$; otherwise, it is \emph{trivial}. The main result of this paper is the following:

\begin{thm}\label{thm:main}
For $r=3$ or $4$, if $T$ is a non-trivial tight $r$-tree, then there exist constants $c_1,c_2>0$ such that, for every positive integer $n$,
$$
c_1n^{r-1}\le R(T,n)\le c_2 n^{r-1}
$$
\end{thm}

We believe this result should extend to all $r$.
\begin{conj}
For $r\ge5$, if $T$ is a non-trivial tight $r$-tree, then there exist constants $c_1,c_2>0$ such that, for every positive integer $n$,
$$
c_1n^{r-1}\le R(T,n)\le c_2 n^{r-1}
$$
\end{conj}

Note that the bounds in~\Cref{thm:main} could be wrong when $T$ is not non-trivial. In particular, when $T$ is a tight $r$-tree with two edges, Phelps and Rödl~\cite{phelps1986steiner} and Rödl and {\v{S}}inajov{\'a}~\cite{rodl1994note} showed that $R(T,n)=\Theta(n^{r-1}/\log n)$.

In fact, it is easy to show that for every tight $r$-tree $T$
$$
\frac{c_1n^{r-1}}{\log n}\le R(T,n)\le c_2n^{r-1}.
$$
The upper bounds come from the random deletion method (see Section 2 for details) and the lower bounds can be obtained from analyzing the Erdös-Rényi random hypergraph using the Lovász Local Lemma (See~\cite[Chapter 5]{alon2016probabilistic} for example). Thus the main difficulty in proving \Cref{thm:main} lies in finding constructions that establish the lower bounds. For $r=3$, we make use of a construction discovered by Cooper and Mubayi~\cite{cooper2017sparse}. For $r=4$, we find a new construction, which is inspired by the Cooper-Mubayi 3-uniform construction. Surprisingly, for $r=5$, natural generalizations of previous constructions do not work anymore; see the first bullet point in the concluding remarks section for more discussions on this.

The rest of this paper is structured as follows. In \Cref{section:Upper}, we prove the upper bounds in \Cref{thm:main}. In \Cref{section:lower_3}, we introduce the Cooper-Mubayi construction, which establishes the 3-uniform lower bounds in \Cref{thm:main}. In \Cref{section:lower_4}, we describe the 4-uniform construction, which implies the 4-uniform lower bounds in \Cref{thm:main}.

\section{Upper bounds}\label{section:Upper}
In this section, we prove the upper bounds in \Cref{thm:main}. Kalai made the following conjecture for tight trees extending a conjecture of Erd\H{o}s and Sós for graph trees.
\begin{conj}[\cite{erdos1964extremal,frankl1987exact}]
Let $r\ge 2$ and let $T$ be a tight $r$-tree with $k$ edges. An $n$-vertex $T$-free $r$-graph has at most $\frac{k-1}{r}\binom{n}{r-1}$ edges.
\end{conj}

Although the conjecture has not been fully resolved, much progress has been made towards it~\cite{furedi2019hypergraphs,furedi2020tight,furedi2020hypergraphs}. For the purpose of this paper, we will only use the following folklore upper bounds, as we don't try to optimize the constant factor.

\begin{prop}[Proposition 5.4, \cite{furedi2015tur}]\label{prop:Turan}
Let $r\ge 2$ and let $T$ be a tight $r$-tree with $k$ edges. An $n$-vertex $T$-free $r$-graph has at most $({k-1})\binom{n}{r-1}$ edges.
\end{prop}

Now we are ready to prove the upper bounds in \Cref{thm:main}
\begin{prop}
Let $r\ge 2$ and let $T$ be a tight $r$-tree with $k$ edges, then
$$
R(T,n)\le 2(k-1)\l(\frac{2en}{r-1}\r)^{r-1}.
$$
\end{prop}

\begin{proof}
Let $N=2(k-1)\l(\frac{2en}{r-1}\r)^{r-1}$ and let $p=\l(\frac{N}{2(k-1)\binom{N}{r-1}}\r)^{\frac{1}{r-1}}$. Let $H$ be an $N$-vertex $T$-free $r$-graph. By \Cref{prop:Turan} we know that $H$ has at most $(k-1)\binom{N}{r-1}$ edges. We randomly select a subset $U\subseteq V(H)$ where each vertex is included independently with probability $p$. Then for each edge in $H[U]$, the induced subgraph of $H$ on $U$, we arbitrarily delete a vertex, and we call the remaining independent set $U'$. The expected size of $U'$ is at least
$$
pN-(k-1)\binom{N}{r-1}p^{r}=\frac{pN}{2}\ge \frac{r-1}{2e}\l(\frac{N}{2(k-1)}\r)^{\frac{1}{r-1}}=n.
$$
Thus we know that $\alpha(H)\ge n$, which comletes the proof.

\end{proof}

\section{The Cooper-Mubayi 3-uniform construction}\label{section:lower_3}
In this section, we introduce the Cooper-Mubayi 3-uniform construction, which implies the 3-uniform lower bounds in \Cref{thm:main}.

\textbf{Cooper-Mubayi construction:} Let $H_3$ be a 3-graph on $[n]^2$, where three vertices $P_0=(x_0,y_0)$, $P_1=(x_1,y_1)$ and $P_2=(x_2,y_2)$ form an edge if and only if $x_0=x_1$, $y_0<y_1$, $y_0=y_2$ and $x_0<x_2$. In other words, $H_3$ consists of triples of grid points that form an ``L shape''(See~\Cref{fig:L} for an example). For each such edge $e$, we call the vertex $P_0$ the \emph{center} of this edge $e$.

\begin{figure}[h]
    \centering
    \begin{tikzpicture}[scale=0.7]
        \tikzstyle{uStyle}=[shape = circle, minimum size = 6.0pt, inner sep = 0pt,
        outer sep = 0pt, draw, fill=white]
        \tikzstyle{fStyle}=[shape = circle, minimum size = 6.0pt, inner sep = 0pt,
        outer sep = 0pt, draw, fill=black]
        \tikzstyle{lStyle}=[shape = rectangle, minimum size = 20.0pt, inner sep = 0pt, outer sep = 2pt, draw=none, fill=none]
        \tikzset{every node/.style=uStyle}

         \foreach \i in {1,2,3,4,5,6}
            \foreach \j in {1,2,3,4,5,6}
                \draw (\i,\j) node (v\i\j){};
        \draw (1.7,1.7) node[lStyle](P0){$P_0$};
        \draw (1.7,4.7) node[lStyle](P1){$P_1$};
        \draw (3.7,1.7) node[lStyle](P2){$P_2$};
        \draw (2,2) node[fStyle](){};
        \draw (2,5) node[fStyle](){};
        \draw (4,2) node[fStyle](){};
        \draw[-latex] (0,0.5) -- (7,0.5);
        \draw (7,0) node[lStyle](){$1^{st}$ coordinate};
        \draw[-latex] (0.5,0) -- (0.5,7);
        \draw (-1.5,6.5) node[lStyle](){$2^{nd}$ coordinate};
    \end{tikzpicture}
    \caption{$P_0$,$P_1$ and $P_2$ form an edge in $H_3$}
    \label{fig:L}
\end{figure}
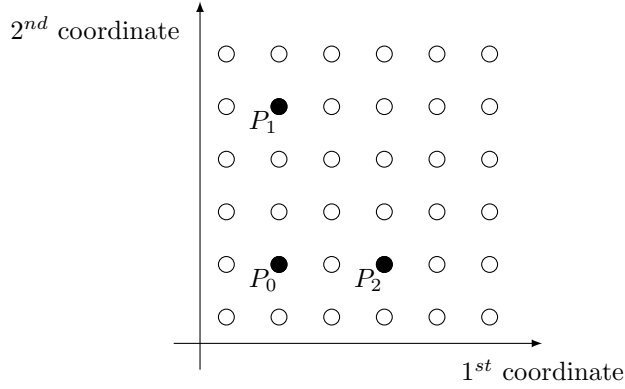

The following proposition is proved in~\cite{cooper2017sparse}. We include it here for completeness.

\begin{prop}\label{prop:independence_3}
Every $2n$-element set in $[n]^2$ contains an edge in $H_3$.
\end{prop}

\begin{proof}
Let $X$ be a $2n$-element set in $[n]^2$. Consider a bipartite graph $B$ on $R\sqcup C$ where $R=\{r_1,r_2,\dots,r_n\}$ and $C=\{c_1,c_2,\dots,c_n\}$ such that $r_i$ and $c_j$ form an edge if and only if $(i,j)\in X$. We can view $R$ and $C$ as the set of rows and the set of columns of the $n\times n$ grid respectively. Note that the number of edges in $B$ is $|X|=2n$ which equals the number of vertices in $B$. So $B$ contains a cycle. In the cycle, there must be three consecutive edges $c_{j'}r_i$, $r_ic_j$, $c_jr_{i'}$ such that $i'>i$ and $j'>j$. Clearly, such a triple of edges corresponds to a triple of points in $X$ which forms an edge in $H_3$.
\end{proof}

\begin{prop}\label{prop:center_3}
For any $e,e'\in H_3$, if $|e\cap e'|=2$, then they have the same center.
\end{prop}

\begin{proof}
Let $e=\{P_0,P_1,P_2\}$ and $e'=\{P'_0,P'_1,P'_2\}$ where $P_0$ and $P'_0$ are the centers of $e_1$ and $e_2$ respectively. Further, for $0\le i\le 2$, let $P_i=(x_i,y_i)$ and let $P'_i=(x'_i,y_i')$. Without loss of generality, suppose $x_1=x_0$, and $x'_1=x'_0$. Then by definition we have $y_1>y_0$, $y'_1>y'_0$, $y_2=y_0$, $y'_2=y'_0$, $x_2>x_0$ and $x'_2>x'_0$. Suppose (for contradiction) $P_0\not=P'_0$. Up to symmetry, the proof splits into two cases.

\textbf{Case 1: $P_0=P'_1$.} In this case, we have $y_1,y_2\ge y_0=y'_1>y'_0, y'_2$, and thus $\{P_1,P_2\}\cap\{P_0',P_2'\}=\emptyset$, which contradicts the condition $|e\cap e'|=2$.

\textbf{Case 2: $P_0\not\in e'$ and $P'_0\not\in e$.} Since $e\cap e'|=2$, we have $e\cap e'=\{P_1,P_2\}=\{P'_1,P'_2\}$. Note that $x_1<x_2$ and $x'_1<x'_2$. This implies that $P_1=P'_1$ and $P_2=P'_2$, and hence $P_0=P'_0$, contradicting the assumption that $P_0\not=P'_0$.
\end{proof}

\begin{prop}\label{prop:forbiden_3}
$H_3$ does not contain any non-trivial tight $3$-tree.
\end{prop}

\begin{proof}
Let $T$ be a tight $3$-tree in $H_3$. Let $e_1,\dots,e_t$ be an order on the edges of $T$ that satisfies \Cref{eq:tight tree}. Note that, $1\le i\le t-1$, $|e_i\cap e_{i+1}|=2$. Thus by \Cref{prop:center_3} all edges $e_1,\dots,e_t$ share the same center, which implies that $T$ is trivial.
\end{proof}

The 3-uniform lower bounds in \Cref{thm:main} follow immediately from \Cref{prop:forbiden_3} and \Cref{prop:independence_3}.

\section{The 4-uniform Construction}\label{section:lower_4}
In this section, we introduce the 4-uniform construction inspired by the Cooper-Mubayi construction, which establishes the 4-uniform lower bounds in \Cref{thm:main}.

For every point $P\in [n]^3$ and every $1\le i\le 3$, let $P(i)$ denote the $i^{th}$ coordinate of $P$. We define directed graphs $T_1, T_2$ and $T_3$ on $[n]^3$ as follows. For convenience, we write a directed edge $(P,Q)$ as $PQ$. For every pair $P,Q\in [n]^3$, 
\begin{itemize}
    \item[(1)] $PQ\in T_1$ if and only if 
$P(1)=Q(1)$
and either
\begin{equation*}
P(2)<Q(2)    
\end{equation*}
or
\begin{equation*}
P(2)=Q(2)~\text{and}~P(3)>Q(3).    
\end{equation*}

\Cref{fig:T1} is a visual illustration of $T_1$, which depicts the possible positions of $Q$ relative to $P$. 

\begin{figure}[h]
    \centering
    \begin{tikzpicture}[scale=0.7]
        \tikzstyle{uStyle}=[shape = circle, minimum size = 6.0pt, inner sep = 0pt,
        outer sep = 0pt, draw, fill=white]
        \tikzstyle{fStyle}=[shape = circle, minimum size = 6.0pt, inner sep = 0pt,
        outer sep = 0pt, draw, fill=black]
        \tikzstyle{sStyle}=[shape = regular polygon,regular polygon sides=4, minimum size = 8.5pt, inner sep = 0pt,
        outer sep = 0pt, draw, fill=black]
        \tikzstyle{lStyle}=[shape = rectangle, minimum size = 20.0pt, inner sep = 0pt, outer sep = 2pt, draw=none, fill=none]
        \tikzset{every node/.style=uStyle}

         \foreach \i in {1,2,3,4,5,6,7}
            \foreach \j in {1,2,3,4,5,6,7}
                \draw (\i,\j) node (v\i\j){};
        \draw (3.7,3.7) node[lStyle](P){$P$};
        \draw (4,4) node[fStyle](){};
        \foreach \i in {5,6,7}
            \foreach \j in {1,...,7}
                \draw (\i,\j) node[sStyle](){};
        \foreach \j in {1,2,3}
            \draw (4,\j) node[sStyle](){};
        \draw[-latex] (0,0.5) -- (8,0.5);
        \draw (8,0) node[lStyle](){$2^{nd}$ coordinate};
        \draw[-latex] (0.5,0) -- (0.5,8);
        \draw (-1.5,7.5) node[lStyle](){$3^{rd}$ coordinate};
    \end{tikzpicture}
    \caption{Here the $\blacksquare$s are the possible positions of $Q$ if $PQ\in T_1$.}
    \label{fig:T1}
\end{figure}
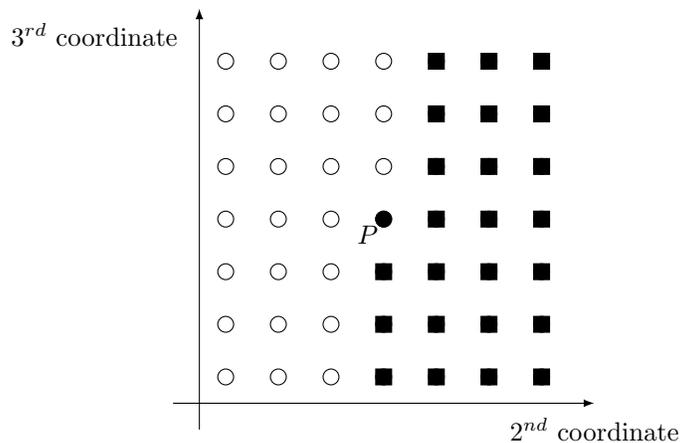

    \item[(2)] $PQ\in T_2$ if and only if 
$P(2)=Q(2)$
and either
\begin{equation*}P(1)<Q(1)    
\end{equation*}
or
\begin{equation*}
P(1)=Q(1)~\text{and}~P(3)<Q(3),    
\end{equation*}

\Cref{fig:T2} is a visual illustration of $T_2$, which depicts the possible positions of $Q$ relative to $P$. 

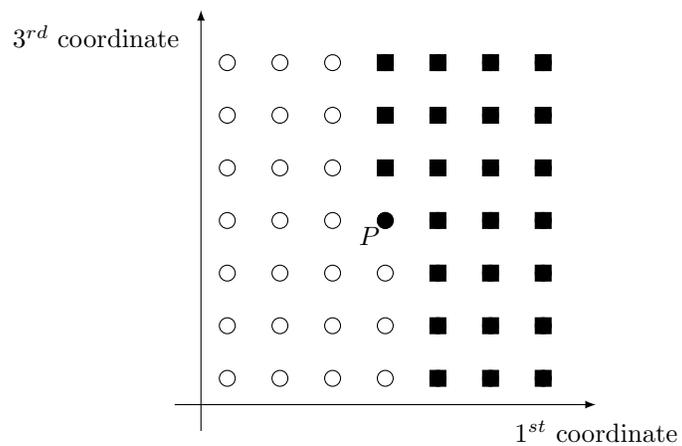
\begin{figure}[h]
    \centering
    \begin{tikzpicture}[scale=0.7]
        \tikzstyle{uStyle}=[shape = circle, minimum size = 6.0pt, inner sep = 0pt,
        outer sep = 0pt, draw, fill=white]
        \tikzstyle{fStyle}=[shape = circle, minimum size = 6.0pt, inner sep = 0pt,
        outer sep = 0pt, draw, fill=black]
        \tikzstyle{sStyle}=[shape = regular polygon,regular polygon sides=4, minimum size = 8.5pt, inner sep = 0pt,
        outer sep = 0pt, draw, fill=black]
        \tikzstyle{lStyle}=[shape = rectangle, minimum size = 20.0pt, inner sep = 0pt, outer sep = 2pt, draw=none, fill=none]
        \tikzset{every node/.style=uStyle}

         \foreach \i in {1,2,3,4,5,6,7}
            \foreach \j in {1,2,3,4,5,6,7}
                \draw (\i,\j) node (v\i\j){};
        \draw (3.7,3.7) node[lStyle](P){$P$};
        \draw (4,4) node[fStyle](){};
        \foreach \i in {5,6,7}
            \foreach \j in {1,...,7}
                \draw (\i,\j) node[sStyle](){};
        \foreach \j in {5,6,7}
            \draw (4,\j) node[sStyle](){};
        \draw[-latex] (0,0.5) -- (8,0.5);
        \draw (8,0) node[lStyle](){$1^{st}$ coordinate};
        \draw[-latex] (0.5,0) -- (0.5,8);
        \draw (-1.5,7.5) node[lStyle](){$3^{rd}$ coordinate};
    \end{tikzpicture}
    \caption{Here the $\blacksquare$s are the possible positions of $Q$ if $PQ\in T_2$.}
    \label{fig:T2}
\end{figure}

\item[(3)] $PQ\in T_3$ if and only if 
$P(3)=Q(3)$
and either
\begin{equation*}
P(2)>Q(2)    
\end{equation*}
or
\begin{equation*}
P(2)=Q(2)~\text{and}~P(1)>Q(1).    
\end{equation*}
\end{itemize}

\Cref{fig:T3} is a visual illustration of $T_3$, which depicts the possible positions of $Q$ relative to $P$. 
\begin{figure}[h]
    \centering
    \begin{tikzpicture}[scale=0.7]
        \tikzstyle{uStyle}=[shape = circle, minimum size = 6.0pt, inner sep = 0pt,
        outer sep = 0pt, draw, fill=white]
        \tikzstyle{fStyle}=[shape = circle, minimum size = 6.0pt, inner sep = 0pt,
        outer sep = 0pt, draw, fill=black]
        \tikzstyle{sStyle}=[shape = regular polygon,regular polygon sides=4, minimum size = 8.5pt, inner sep = 0pt,
        outer sep = 0pt, draw, fill=black]
        \tikzstyle{lStyle}=[shape = rectangle, minimum size = 20.0pt, inner sep = 0pt, outer sep = 2pt, draw=none, fill=none]
        \tikzset{every node/.style=uStyle}

         \foreach \i in {1,2,3,4,5,6,7}
            \foreach \j in {1,2,3,4,5,6,7}
                \draw (\i,\j) node (v\i\j){};
        \draw (3.7,3.7) node[lStyle](P){$P$};
        \draw (4,4) node[fStyle](){};
        \foreach \j in {1,2,3}
            \foreach \i in {1,...,7}
                \draw (\i,\j) node[sStyle](){};
        \foreach \i in {1,2,3}
            \draw (\i,4) node[sStyle](){};
        \draw[-latex] (0,0.5) -- (8,0.5);
        \draw (8,0) node[lStyle](){$1^{st}$ coordinate};
        \draw[-latex] (0.5,0) -- (0.5,8);
        \draw (-1.5,7.5) node[lStyle](){$2^{nd}$ coordinate};
    \end{tikzpicture}
    \caption{Here the $\blacksquare$s are the possible positions of $Q$ if $PQ\in T_3$.}
    \label{fig:T3}
\end{figure}
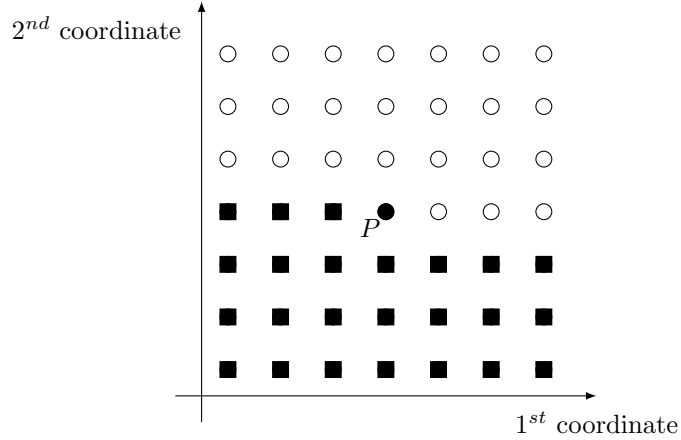

The following proposition can be easily checked by looking at Figures \ref{fig:T1}, \ref{fig:T2} and \ref{fig:T3}.
\begin{prop}\label{prop:disjoint}
If $PP_1\in T_1$, $PP_2\in T_2$ and $PP_3\in T_3$, then $P_1$, $P_2$ and $P_3$ are distinct.
\end{prop}

\textbf{The 4-uniform construction:} Let $H_{4}$ be the $4$-graph on $[n]^3$ such that four distinct points $P_0,P_1,P_2$ and $P_3$ form an edge $e$ if and only if, $\forall 1\le i\le 3$, $P_0P_i\in T_i$, and we call $P_0$ the {\em center} for this edge $e$.

\begin{prop}\label{prop:independence_4}
Every $10n$-element set in $[n]^3$ contains an edge in $H_{4}$.
\end{prop}

\begin{proof}
Let $X$ be a $10n$-element set in $[n]^3$. For $1\le i\le 3$ and $1\le j\le n$, let $X_{ij}\subseteq [n]^3$ be the set consisting of all points $P$ such that $P(i)=j$. We iteratively delete elements in $X_{ij}$ from $X$ if $|X\cap X_{ij}|\le 3$. Since the number of such $X_{ij}$ is at most $3n$, we know that the number of deleted vertices is at most $9n$.

Thus after all deletions, we still have a nonempty set $X'$. We want to find a $P_0\in X'$ such that $P_0$ has non-zero outdegree in all digraphs $T_i[X']$. To this end, for every $1\le i\le 3$, we let $B_i$ be the set of elements $P$ such that $P$ has outdegree 0 in $T_i[X']$. By definition, for every pair $ij$, if $X'\cap X_{ij}\not=\emptyset$, then $|X'\cap X_{ij}|\ge 4$. Note that $T_i[X'\cap X_{ij}]$ is a tournament, thus there is at most one vertex in $X'\cap X_{ij}$ with outdegree 0 in $T_i[P']$. So we have $|B_i|\le |X'|/4$, and hence 
$$
|X'\setminus(\cup_{i=1}^3 B_i)|\ge |X'|-\sum_{i=1}^3|B_i|\ge |X'|/4>0.
$$
Thus there exists a $P_0\in X'$ such that $P$ has non-zero outdegree in all digraphs $T_i[X']$. So we can pick $P_1$, $P_2$ and $P_3\in X'$ such that $P_0P_1\in T_1$, $P_0P_2\in T_2$ and $P_0P_3\in T_3$. By \Cref{prop:disjoint}, we know that $P_1$, $P_2$ and $P_3$ are distinct. Therefore, $P_0$,$P_1$,$P_2$ and $P_3$ form an edge in $H_4$.
\end{proof}

\begin{prop}\label{prop:center_4}
For any $e,e'\in H_4$, if $|e\cap e'|=3$, then they have the same center.
\end{prop}
\begin{proof}
Let $e=\{P_0,P_1,P_2,P_3\}$ and $e'=\{P_0',P_1',P_2',P_3'\}$ where $P_0P_i, P_0'P'_i\in T_i$ for every $1\le i\le 3$. Suppose (for contradiction) that $|e\cap e'|=3$ and that $e$ and $e'$ have different centers, i.e. $P_0\not=P'_0$. Up to symmetry, the proof splits into the following four cases:

\noindent\textbf{Case 1: $P_0\not\in e'$ and $P_0'\not\in e$.}

In this case, we have $\{P_1,P_2,P_3\}=\{P'_1,P'_2,P'_3\}$. We will show that $P_1=P_1'$, $P_2=P_2'$ and $P_3=P_3'$. 

Define the function $F_{21}:[n]^3\rightarrow \Z$ such that 
\begin{equation}\label{eq:21}
    F_{21}(P)=P(2)(n+1)+P(1)
\end{equation}

It is easy to check that, by definition, for every $t\in \{1,2\}$, we have $F_{21}(P_3)<F(P_t)$ and $F_{21}(P'_3)<F(P'_t)$. Thus $P_3=P_3'$. 

Similarly, define the function $F_{2(-1)(-3)}:[n]^3\rightarrow \Z$ such that 
\begin{equation}\label{eq:2-1-3}
F_{2(-1)(-3)}(P)=P(2)(n+1)^2-P(1)(n+1)-P(3).
\end{equation}
It is easy to check that, by definition, $F_{2(-1)(-3)}(P_2)<F(P_1)$ and $F_{2(-1)(-3)}(P'_2)<F(P'_1)$. Thus we have $P_2=P_2'$ and $P_1=P_1'$. 

Note that $P(i)=P_i(i)=P'_i(i)=P'(i)$ for every $1\le i\le 3$. This implies that $P=P'$, which contradicts our assumption that $P\not=P'$.  

\medskip

\noindent\textbf{Case 2: $P_0'=P_1$.}

We will first show that $P_1'\not\in e$. 

Define the function $F_{2(-3)}:[n]^3\rightarrow \Z$ such that 
\begin{equation}\label{eq:2-3}
F_{2(-3)}(P)=P(2)(n+1)-P(3).   
\end{equation}
Then by definition 
$$F_{2(-3)}(P_1')>F_{2(-3)}(P_0')=F_{2(-3)}(P_1)>F_{2(-3)}(P_0).$$ 
Thus $P_1'\not=P_0$.

Recall \Cref{eq:2-1-3}. By definition, 
$$F_{2(-1)(-3)}(P'_1)>F_{2(-1)(-3)}(P_0')=F_{2(-1)(-3)}(P_1)>F_{2(-1)(-3)}(P_2).$$ 
Thus $P'_1\not=P_2$. 

Recall \Cref{eq:21}. By definition, 
$$F_{21}(P'_1)\ge F_{21}(P'_0)=F_{21}(P_1)>F_{21}(P_3).$$ 
Thus $P'_1\not=P_3$. We conclude that $P'_1\not\in e$, and hence $\{P'_2,P'_3\}\subseteq e$.

Note that $P'_0(2)=P_1(2)\ge P_0(2)$. So we either have $P'_0(2)>P_0(2)$ or $P'_0(2)=P_0(2)$. If $P'_0(2)>P_0(2)$, then $P'_2(2)=P'_0(2)>P_0(2)=P_2(2)\ge P_3(2)$, which implies that $P'_2\not=P_0,~P_2$ or $P_3$. Further, $P'_2\not=P'_0=P_1$, so we have $P'_2\not\in e$, a contradiction. Thus we have $P'_0(2)=P_1(2)=P_0(2)$. Then by definition of $T_1$ we must have $P_1(3)<P_0(3)$, and hence $P'_3(3)=P_0'(3)=P_1(3)<P_0(3)=P_3(3)$, which implies $P'_3\not=P_0$ or $P_3$. Moreover, recall \Cref{eq:21}, by definition
$$
F_{21}(P'_3)<F_{21}(P')=F_{21}(P_0)\le F_{21}(P_2).
$$
Thus $P'_3\not=P_2$. Further, $P'_3\not=P'_0=P_1$. We conclude that $P'_3\not\in e$, again a contradiction.

\noindent\textbf{Case 3: $P'_0=P_2$}

We will first show that $P_2'\not\in e$. 

Define the function $F_{13}:[n]^3\rightarrow \Z$ such that 
\begin{equation}\label{eq:13}
F_{13}(P)=P(1)(n+1)+P(3).   
\end{equation}
Then by definition 
$$F_{13}(P_2')>F_{13}(P_0')=F_{13}(P_2)>F_{13}(P_0).$$ 
Thus $P_2'\not=P_0$.

Recall \Cref{eq:2-1-3}. By definition, 
$$F_{2(-1)(-3)}(P'_2)<F_{2(-1)(-3)}(P_0')=F_{2(-1)(-3)}(P_2)<F_{2(-1)(-3)}(P_1).$$ 
Thus $P'_2\not=P_1$. 

Recall \Cref{eq:21}. By definition, 
$$F_{21}(P'_2)\ge F_{21}(P_0')=F_{21}(P_2)\ge F_21(P_0)>F_{21}(P_3).$$ 
Thus $P'_2\not=P_3$. Further, $P'_2\not=P'_0=P_2$. We conclude that $P'_2\not\in e$, and hence $\{P'_1,P'_3\}\subseteq e$.

Note that $P'_0(1)=P_2(1)\ge P_0(1)$. So we either have $P'_0(1)>P_0(1)$ or $P'_0(1)=P_0(1)$. If $P'_0(1)>P_0(1)$, then $P'_1(1)=P'_0(1)>P_0(1)=P_1(1)$, which implies that $P'_1\not=P_0$ or $P_1$. Moreover, recalling \Cref{eq:21}, we have $F_{21}(P_1')\ge F_{21}(P_0')=F_{21}(P_2)\ge F_{21}(P_0)>F_{21}(P_3)$, which implies $P'_1\not=P_3$. Further, $P'_1\not=P'_0=P_2$. So we have $P'_1\not\in e$, a contradiction. Thus we have $P'_0(1)=P_2(1)=P_0(1)$. Then by definition of $T_2$ we must have $P_2(3)>P_0(3)$, and hence $P'_3(3)=P'_0(3)=P_2(3)>P_0(3)=P_3(3)$, which implies $P'_3\not=P_0$ or $P_3$. Moreover, recalling \Cref{eq:21}, by definition
$$
F_{21}(P'_3)<F_{21}(P_0')=F_{21}(P_2)=F_{21}(P_0)\le F_{21}(P_1).
$$
Thus $P'_3\not=P_1$. Further, $P'_3\not=P'_0=P_2$. We conclude that $P'_3\not\in e$, again a contradiction.

\noindent\textbf{Case 4: $P'_0=P_3$.}

Recall \Cref{eq:21}, by definition,
$$
F_{21}(P'_3)<F_{21}(P_0')=F_{21}(P_3)<\min\{F_{21}(P_0), F_{21}(P_1),F_{21}(P_2)\}.
$$
Thus $P_3'\not \in e$, and hence $\{P_1',P_2'\}\subseteq e$.

Note that $P'_0(2)=P_3(2)\le P_0(2)$. So we have either $P'_0(2)<P_0(2)$ or $P'_0(2)=P_0(2)$. If $P'_0(2)<P_0(2)$, then $P'_2(2)=P_0'(2)<P_0(2)=P_2(2)\le P_1(2)$, which implies $P'_2\not=P_0,~P_1$ or $P_2$. Further, $P'_2\not=P'_0=P_3$. So we have $P'_2\not\in e$, a contradiction. Thus we have $P_0'(2)=P_3(2)=P_0(2)$. By definition of $T_3$ we must have $P_0'(1)=P_3(1)<P_0(1)$, and hence $P'_1(1)=P_0'(1)<P_0(1)=P_1(1)\le P_2(1)$, which implies $P'_1\not=P_0,~P_1$ or $P_2$. Further, $P'_1\not=P'_0=P_3$. We conclude that $P'_1\not\in e$, again a contradiction.
\end{proof}

\begin{prop}\label{prop:forbiden_4}
$H_4$ does not contain any non-trivial tight $4$-tree.
\end{prop}

\begin{proof}
Let $T$ be a tight $4$-tree in $H_4$. Let $e_1,\dots,e_t$ be an order on the edges of $T$ that satisfies \Cref{eq:tight tree}. Note that, $1\le i\le t-1$, $|e_i\cap e_{i+1}|=3$. Thus by \Cref{prop:center_4} all edges $e_1,\dots,e_t$ share the same center, which implies that $T$ is trivial.
\end{proof}

The 4-uniform lower bounds in \Cref{thm:main} follow immediately from \Cref{prop:forbiden_4} and \Cref{prop:independence_4}.

\section{Concluding remarks}
\begin{itemize}
    \item Observe that in our 4-uniform construction, $T_i$, when restricted to points with fixed $i^{th}$ coordinate, is actually defined by a signed lexicographical order. Thus a natural generalization to a 5-uniform construction is, for $1\le i\le 4$, we let $T_i$ be directed graphs on $[n]^4$ defined by one of the $2^33\,!=48$ signed lexicographical orders, and then define the edges in our 5-graph to be the 5-tuples $\{P_0,P_1,P_2,P_3,P_4\}$ such that $P_0P_i\in T_i$ for all $1\le i\le 4$. However, after checking with a computer program, we conclude that none of such constructions work for the 5-uniform problem. Indeed, we have checked that none of the 5-uniform constructions as described above satisfy the 5-uniform variant of \Cref{prop:center_4}.
    \item Note that our construction is in some sense not symmetric. The asymmetry here is in fact necessary. To see this, let's consider changing the definitions of $T_1$, $T_2$ and $T_3$ into the following more symmetric ones: we define the directed graphs $T_1, T_2$ and $T_3$ on $[n]^3$ as follows. For every pair of points $P,Q\in [n]^3$, 
\begin{itemize}
    \item[(1)] $PQ\in T_1$ if and only if 
$P(1)=Q(1)$
and either
\begin{equation*}
P(2)<Q(2)    
\end{equation*}
or
\begin{equation*}
P(2)=Q(2)~\text{and}~P(3)>Q(3),    
\end{equation*}

    \item[(2)] $PQ\in T_2$ if and only if 
$P(2)=Q(2)$
and either
\begin{equation*}P(3)<Q(3)    
\end{equation*}
or
\begin{equation*}
P(3)=Q(3)~\text{and}~P(1)>Q(1),    
\end{equation*}

\item[(3)] $PQ\in T_3$ if and only if 
$P(3)=Q(3)$
and either
\begin{equation*}
P(1)<Q(1)    
\end{equation*}
or
\begin{equation*}
P(1)=Q(1)~\text{and}~P(2)>Q(2).    
\end{equation*}
\end{itemize}
Let $P_0=(0,0,0)$, $P_1=(0,1,2)$, $P_2=(2,0,1)$, $P_3=(1,2,0)$, $P_0'=(1,1,1)$. Note that $P_0P_1,~P'_0P_3\in T_1$, $P_0P_2,~P'_0P_1\in T_2$ and $P_0P_3,~P'_0P_2\in T_3$. Thus both $\{P_0,P_1,P_2,P_3\}$ and $\{P_0',P_1,P_2,P_3\}$ are edges in $H_4$, but they have different centers, contradicting \Cref{prop:center_4}.
\item Note that the graph Ramsey number for tree is related to the number of edges in the tree. For tight $r$-tree with $k$ edges, in \Cref{section:Upper} we have shown that $R(n,T)=O(kn^{r-1})$. It would be interesting if we can prove a corresponding lower bound.
\begin{ques}
Let $T$ be a non-trivial tight $r$-tree with $k$ edges. Is it true that $R(T,n)=\Omega\l(kn^{r-1}\r)$?
\end{ques}

\end{itemize}

\section*{Acknowledgment}
The author would like to thank Xiaoyu He and Dhruv Mubayi for helpful discussions.

\bibliographystyle{abbrv}
\bibliography{refs}

\end{document}